\documentclass[12pt]{article}
\usepackage{amssymb}
\usepackage{amsthm}
\usepackage{color}

\usepackage{amsmath,dsfont}

\newcommand{\E}{\mathbf E\,}
\newcommand{\R}{\mathbb R}
\newcommand{\N}{\mathbb N}
\newcommand{\I}{\mathbb I}

\newcommand{\Prob}[1]{\mathbf{P}\left\{#1\right\}}

\newcommand{\Rr}{\mathbb{R}^r}
\newcommand{\Nr}{\mathbb{N}^r}
\newcommand{\one}{\mathbf{1}}
\newcommand{\sT}{\mathcal{T}}

\theoremstyle{plain} \newtheorem{theorem}{Theorem}[section]
\theoremstyle{plain} 
\theoremstyle{plain} \newtheorem{lemma}[theorem]{Lemma}
\theoremstyle{plain} \newtheorem{corollary}[theorem]{Corollary}
\theoremstyle{definition} 
\theoremstyle{definition} 
\theoremstyle{remark} \newtheorem{remark}[theorem]{Remark}
\theoremstyle{remark} 

\def\vect#1{{\boldsymbol{#1}}}
\def\n{\vect{n}}
\def\k{\vect{k}}
\def\m{\vect{m}}
\def\x{\vect{x}}
\def\b{\vect{b}}

\def\var{\operatorname{var}}
\def\increment#1{\Delta[#1]}
\def\incrementbigg#1{\Delta\bigg[#1\bigg]}

\def\eps{\varepsilon}

\def\field#1{\{#1_{\n},\n\in\Nr\}}

\def\bk{b_{\k}}
\def\bn{b_{\n}}

\def\xk{x_{\k}}
\def\xn{x_{\n}}
\def\Sn{S_{\n}}\def\Sk{S_{\k}}

\def\Cn{C_{\n}}

\def\Zk{Z_{\k}}\def\Zn{Z_{\n}} 
\def\an{a_{\n}}\def\ak{a_{\k}}

\def\Seq#1#2#3{\left\{#1_#2^{\csname#3\endcsname}\,\right\}_{#2=1}^\infty}
\def\YI#1{Y_{1,\csname#1\endcsname}}
\def\Yi#1{Y_{i,\csname#1\endcsname}}

\def\tr#1#2{#1^{#2}}

\let\phi\varphi

\usepackage{fancybox}
\newlength{\querylen}
\begin{document}

\title{Moment conditions in strong laws of large numbers for
    multiple sums and random measures\thanks{Supported by
      Swiss National Science Foundation Scopes Programme Grant
      IZ7320\_152292}}

\author{Oleg Klesov\footnote{
Department of Mathematical Analysis and Probability Theory,
National Technical University of Ukraine (KPI),
Peremogy avenue, 37,
Kyiv, 03056,
Ukraine, klesov@matan.kpi.ua} and 
Ilya Molchanov\footnote{
University of Bern,
 Institute of Mathematical Statistics and Actuarial Science, 
 Sidlerstrasse 5,
 CH-3012 Bern, 
ilya.molchanov@stat.unibe.ch}}

\date{}

\maketitle

\begin{abstract}
  The validity of the strong law of large numbers for multiple sums
  $\Sn$ of independent identically distributed random variables $\Zk$,
  $\k\leq\n$, with $r$-dimensional indices 
  is equivalent to the integrability of $|Z|(\log^+|Z|)^{r-1}$, where
  $Z$ is the generic summand.  We consider the strong law of large
  numbers for more general normalizations, without assuming that the
  summands $\Zk$ are identically distributed, and prove a multiple sum
  generalization of the Brunk--Prohorov strong law of large
  numbers. In the case of identical finite moments of order $2q$ with
  integer $q\geq1$, we show that the strong law of large numbers holds
  with the normalization $(n_1\cdots n_r)^{1/2}(\log n_1\cdots\log
  n_r)^{1/(2q)+\eps}$ for any $\eps>0$.

  The obtained results are also formulated in the setting of ergodic
  theorems for random measures, in particular those generated by
  marked point processes.
\end{abstract}

\section{Introduction}
\label{sec:introduction}

Let $r\ge1$ be an integer number and let $\Nr$ denote the set of
$r$-dimensional vectors with positive integer coordinates.  Elements
of $\Nr$ are denoted by $\k$, $\n$ etc.  The inequality $\k\le\n$ is
defined coordinatewisely, that is $k_i\le n_i$, $1\le i\le r$, where
$\k=(k_1,\dots,k_r)$ and $\n=(n_1,\dots,n_r)$. Denote
$|\n|=n_1\cdots n_r$. Then, $|\n|\to\infty$ means that the
maximum of all coordinates of $\n$ converges to infinity and so is
called max-convergence or product convergence, see
\cite{kles14}. Furthermore, $\n\to\infty$ means that all components of
$\n$ converge to infinity, that is $\min(n_1,\dots,n_r)\to\infty$, it
is called the min-convergence in \cite{kles14}.

Consider an array $\field b$ of positive numbers indexed by $\Nr$ such
that $\bn\to\infty$ as $|\n|\to\infty$. Define partial sums of
random variables $\field Z$ by
\begin{displaymath}
  \Sn=\sum_{\k\leq\n} \Zk,\quad \n\in\Nr.
\end{displaymath}
The random field $\field Z$ is said to satisfy the strong law of large
numbers with the normalization $\field b$ if $\Zn$ is integrable for
all $\n$ and 
\begin{equation}
  \label{eq:slln}
  \frac{1}{\bn}(\Sn-\E\Sn)\to0 \qquad\text{a.s. as }\;
  |\n|\to\infty.
\end{equation}
If all $\Zn$'s are centered or are not integrable, the validity of the
strong law of large numbers means that
\begin{equation}
  \label{eq:slln:E=0}
  \frac{1}{\bn}\Sn\to0 \qquad\text{a.s. as }\;
  |\n|\to\infty.
\end{equation}
It is easy to see that $\bn$ should grow faster than $\sqrt{\n}$. If
$\field Z$ are independent copies of a centered random variable $Z$,
then \eqref{eq:slln:E=0} for $\bn=|\n|$ becomes the strong law of
large numbers for multiple sums, which holds if and
only if $\E[|Z|(\log^+|Z|)^{r-1}]<\infty$, see \cite{smy73}. Here
$\log^+t$ denotes the positive part of $\log t$. If $\bn$ grows faster
than $|\n|$, the corresponding results are variants of the
Marcinkiewicz--Zygmund law. In this paper we present a whole spectrum
of such results exploring relations between the strength of the moment
conditions and the growth rate of the sequence of normalising
constants. In particular, we show that imposing sufficiently strong
moment assumptions makes it possible to bring the normalising factors
to $\bn=|\n|^{1/2}(\log n_1\cdots\log n_r)^\eps$ for any $\eps>0$.  

The strong law of large numbers was used in \cite{smy75} to derive the
ergodic theorem for sums generated by marked point processes. We first
provide an alternative proof (that gives a stronger result under
weaker conditions) of the strong law of large numbers claimed in
\cite{smy75} to follow from the multivariate analogue of the Kronecker
lemma. As we show in Section~\ref{sec:kron-lemma-mult-2}, this lemma
holds only in the nonnegative case. Indeed, we provide a
counterexample to a ``natural'' generalization of the Kronecker lemma
which invalidates the proof of \cite[Th.~2.1.1]{smy75}.

Section~\ref{sec:conv-mult-sums} contains several strong laws of large
numbers for multiple sums of not identically distributed random
variables that combine moment conditions on the summands with not so
fast growing normalising constants. Along the same line, we generalize
the Brunk--Prohorov criterion for the validity of the strong law of
large numbers known for the case of univariate sums, see
\cite{brun48,proh50}. In case of i.i.d. summands, the conditions
simplify substantially. 

Section~\ref{sec:an-ergodic-theorem} rephrases the results from
Section~\ref{sec:conv-mult-sums} for random measures, in particularly,
those generated by marked point processes.

\section{Strong laws of large numbers for multiple sums}
\label{sec:conv-mult-sums}

\subsection{Conditions on moments of order up to $2$}
\label{sec:cond-moments-order}

The field $\field b$ is said to be monotonic if $\bk\leq\bn$ for
$\k\leq\n$ coordinatewisely.  Define the increments of $\field b$ by
\begin{displaymath}
  \Delta[\bn]=\sum_{\m=(m_1,\dots,m_r)\in\{0,1\}^r} 
  (-1)^{m_1+\cdots+m_r} \b_{\n-\m},
\end{displaymath}
where the array $\field b$ is extended for indices with
\emph{non-negative} coordinates by letting $\bn=0$ if at least one of
the coordinates of $\n$ vanishes. The non-negativity of
$\Delta[\bn]$ for all $\n$ is a stronger condition than the
monotonicity of $\field b$.

The following Theorem~\ref{thr:211} appears as \cite[Th.~2.1.1]{smy75} and was
announced first in \cite{smy73}. However, it was formulated in the
particular case $\n\to\infty$ and assuming the non-negativity of
increments $\increment{\bn}$ for the weights. In order to deduce the
strong law of large numbers from the convergence of random multiple
series, it relied on the Kronecker lemma for multiple sums that was
mentioned as a ``simple generalization'' of the univariate case in
\cite[p.~116]{smy75}. It will be explained in
Section~\ref{sec:kron-lemma-mult-2} that such a generalization holds
only assuming that the summands are non-negative, and so the proof of
\cite[Th.~2.1.1]{smy75} was not complete. We suggest an
alternative proof that derives the strong law of large numbers under
the max-convergence $|\n|\to\infty$, and for this it is unavoidable
to assume that
\begin{equation}
  \label{bn->infty}
  \text{$\bn\to\infty$ as $|\n|\to\infty$}
\end{equation}
instead of
$\n\to\infty$ in~\cite{smy75}. The one-dimensional case is considered
in \cite{faz:kles00}.

Note that the convergence of multiple series $\sum_{\n\in\Nr}\an$ is
always understood as the convergence of their partial sums
$\sum_{\k\le\n}\ak$ as $\n\to\infty$.

\begin{theorem}
  \label{thr:211}
  Assume that $\field b$ is monotonic. 
  Let $\phi$ be a positive even continuous function on $\R$ such that
  $x^{-1}\phi(x)$ is non-decreasing and $x^{-2}\phi(x)$ is
  non-increasing for $x>0$.  If $\field Z$ are independent
  centered random variables such that
  \begin{equation}
    \label{eq:r>1}
    \sum_{\k\in\Nr}\frac{\E\phi(\Zk)}{\phi(\bk)}<\infty,
  \end{equation}
  then the series $\sum_{\k\in\Nr}\Zk/\bk$ converges almost surely and
  \eqref{eq:slln} holds.
\end{theorem}
\begin{proof}
  Given the conditions imposed on $\phi$, 
  it follows from the proof of \cite[Th.~6.4]{petrov95} that 
  \begin{align*}
    \Prob{|X|\ge b} &\le \frac {\E\phi(X)}{\phi(b)},
    \\
    \big|\E(X\one_{\{|X|<b\}})\big| &\le \frac {b}{\phi(b)}\E\phi(X),
    \\
    \E(X^2\one_{\{|X|<b\}}) &\le \frac {b^2}{\phi(b)} \E\phi(X).
  \end{align*}
  for each centered random variable $X$ with $\E\phi(X)<\infty$.
  Let $\tr Xt$ be the truncation of a
  random variable $X$ at the level $t>0$, namely $\tr
  Xt=X\one_{\{|X|<t\}}$.  
  Condition~\eqref{eq:r>1} together with the latter three inequalities
  imply that the following three series
  \begin{equation}
    \label{eq:3}
    \sum_{\n\in\Nr}\Prob{|\Zn|\ge \bn}, \qquad
    \sum_{\n\in\Nr}\left|\frac{\E\Zn^{\bn}}{\bn}\right|, \qquad \sum_{\n\in\Nr}
    \frac{\var\Zn^{\bn}}{\bn^2}
  \end{equation}
  converge.  We conclude from the convergence of the first series that
  \begin{equation}
    \label{io}
    \Prob{\Zn\ne\tr{\Zn}{\bn} \text{ infinitely often}}=0
  \end{equation}
  by the Borel--Cantelli lemma.  By \cite[Th.~5.7]{kles14},
  the convergence of the second and third series implies that $\sum
  \tr{\Zn}{\bn}/\bn$ converges almost surely, whence $\sum \Zn/\bn$
  converges almost surely in view of~\eqref{io}.

  Further, \cite[Cor.~8.1]{kles14}, \cite[Cor.~2.1]{kles98}, and
  convergence of the third series in~\eqref{eq:3} yield 
  \begin{equation}
    \label{SLLN-for-truncated}
    \frac1{\bn}\sum_{\k\le\n}(\Zk^{\bk}-\E \Zk^{\bk})\to0\qquad
    \text{a.s. as }\; |\n|\to\infty.
  \end{equation}
  The convergence of the second series in~\eqref{eq:3} together with
  a version of the Kronecker lemma (which is of independent interest
  and appears as Lemma~\ref{kron-gen} in the last section of the paper) imply that
  \begin{displaymath}
    \frac1{\bn}\sum_{\k\le\n}\left|\E\Zk^{\bk}\right|\to0 \qquad
    \text{as }\; |\n|\to\infty.
  \end{displaymath}
  Combining this result with~(\ref{SLLN-for-truncated}) and~(\ref{io})
  yields~\eqref{eq:slln}.
\end{proof}

Condition~\eqref{eq:r>1} is not optimal for i.i.d. $\field
Z$. For instance if $\bn=|\n|$, then it would require the
integrability of $Z_1(\log^+ |Z_1|)^{r+\eps}$ for some $\eps>0$, whereas
the optimal condition is the integrability of $Z_1(\log^+ |Z_1|)^{r-1}$,
see \cite{smy73}. 

The following result is obtained by letting $\phi(t)=|t|^\alpha$
in Theorem~\ref{thr:211}. 

\begin{corollary}
  \label{cor:5.1}
  Let $1\le\alpha\le2$.  If $\field b$ is monotonic and \eqref{bn->infty} holds,
 $\field Z$ are
  independent centered random variables with $\E|\Zn|^\alpha<\infty$ for
  all $\n$, and
  \begin{equation}
    \label{eq:5-first}
    \sum_{\n\in\Nr}\frac {\E|\Zn|^\alpha}{\bn^\alpha}<\infty,
  \end{equation}
  then \eqref{eq:slln:E=0} holds.
\end{corollary}

\subsection{Brunk--Prohorov theorem for multiple sums}
\label{sec:brunk-proh-theor}

The following variant of the strong law of large numbers involves
higher moments.

\begin{theorem}
  \label{alpha>=2}
Let $q\geq1$ be an integer.  Assume that $\field b$ is monotonic and \eqref{bn->infty} holds,
  $\field Z$ are independent centered random variables with
  $\E\Zn^{2q}<\infty$ for all $\n$, and
  \begin{equation}
    \label{eq:4}
    \sum_{\n\in\Nr} \frac {\an}{\bn^{2q}}<\infty
  \end{equation}
  for 
  \begin{equation}
    \an=\incrementbigg{|\n|^{q-1}\sum_{\k\le\n}\E\Zk^{2q}}.
    \label{sequence-if-r>1}
  \end{equation}
  Then \eqref{eq:slln:E=0} holds. 
\end{theorem}
\begin{proof}
  An analogue of Doob's inequality 
  for multiple sums \cite{wic69} yields that 
  \begin{displaymath}
    \E\max_{\k\le\n} \Sk^{2q}
    \le C'\E\Sn^{2q}
    \le C|\n|^{q-1}\sum_{\k\le\n}\E \Zk^{2q}
  \end{displaymath}
  for some constants $C'$ and $C$, where the second inequality follows
  by iterating the Dharmadhikari--Fabian--Jogdeo inequality 
  \cite{dhar:fab:jog68} several times in order to reduce the
  dimensionality of the summation index. 
  
  Without loss of generality, assume that $\bn\ge1$ for all
  $\n\in\Nr$.  Fix $t\ge0$ and consider
  \begin{math}
    A_t=\{\n\in\Nr:\;\bn\le 2^t\}.
  \end{math}
  Pick $\n_t$ such that $A_t\subseteq\{\k:\k\le\n_t\}$.  For
  $\k\leq\n_t$, let $\Zk^*=\Zk$ if $\k\in A_t$ and $\Zk^*=0$
  otherwise, and denote their multiple sums by
  $\Sn^*=\sum_{\k\le\n}\Zk^*$.  These auxiliary random variables are
  needed to convert the summation domain to a rectangle. Finally, let 
  \begin{displaymath}
    \an^*=\incrementbigg{|\n|^{q-1}\sum_{\k\le\n}\E(\Zk^*)^{2q}}
    =\incrementbigg{|\n|^{q-1}\sum_{\k\le\n}\E\Zk^{2q}}, \qquad \n\in A_t.
  \end{displaymath}
  Note that $\an^*\geq0$, being the increment of the product of two
  monotonic fields, see \cite[Lemma~8.3]{kles14}.
  For $\n\notin A_t$, let $\an^*=0$.  Reasoning as above, we obtain 
  \begin{displaymath}
    \E{\max_{\n\in A_t}\Sn^{2q}} \le \E{\max_{\n\le\n_t} (\Sn^*)^{2q}}
    \le C \sum_{\n\le\n_t} \an^* = C \sum_{\n\in A_t}\an.
  \end{displaymath}
  The proof is completed by referring to \cite[Th.~8.3]{kles14}.
\end{proof}

\begin{remark}
  An analogue of Theorem~\ref{alpha>=2} for cumulative sums,
  i.e. in dimension $r=1$, goes back to Brunk \cite{brun48} and Prohorov
  \cite{proh50}. They proved that, if $\zeta_n=\xi_1+\dots+\xi_n$ are
  cumulative sums of independent random variables $\{\xi_i,i\geq1\}$
  and
  \begin{equation}
    \label{eq:6}
    \sum_{k=1}^\infty \frac{\E\xi_k^{2q}}{k^{q+1}}<\infty
  \end{equation}
  for $q\geq1$, then $(\zeta_n-\E\zeta_n)/n\to0$ a.s. as $n\to\infty$.
  The choice $q=1$ yields validity of 
  the Kolmogorov strong law of large numbers.
  A similar result can be proved for a normalization by an arbitrary
  increasing and unbounded sequence $\{b_k,k\geq1\}$ of positive
  numbers. Then 
  \begin{equation}
    \label{eq:7}
    \sum_{k=2}^\infty \frac{a_k}{b_k^{2q}}<\infty
  \end{equation}
  substitutes~\eqref{eq:6} as a sufficient condition for the strong
  law of large numbers, where 
  \begin{align}
    a_k&=k^{q-1}\sum_{j=1}^k\E\xi_j^{2q}-(k-1)^{q-1}
    \sum_{j=1}^{k-1}\E\xi_j^{2q}\notag
    \\
    &=k^{q-1}\E\xi_k^{2q}+O(1)k^{q-2}\sum_{j=1}^{k-1}\E\xi_j^{2q}.
    \label{sequence-if-r=1}
  \end{align}
  This sequence coincides with that given by~\eqref{sequence-if-r>1}
  in dimension $r=1$, where Theorem~\ref{alpha>=2} becomes the
  Brunk--Prohorov strong law of large numbers. 

  For this reason, Theorem~\ref{alpha>=2} can be called
  the Brunk--Prokhorov theorem for multiple sums. Other
  generalizations of the Brunk--Prokhorov theorem are obtained in
  \cite{lag09} and \cite{son:than13}. 
\end{remark}

\begin{corollary}
  \label{cor:equal-mean}
  Assume that independent centered random variables $\{\Zn,\n\in\Nr\}$
  have the same finite moment of order $2q$.  If $\field b$ is
  monotonic and \eqref{bn->infty} holds, then~\eqref{eq:slln:E=0}
  follows from
  \begin{equation}
    \label{condition-for-(2,infty)]}
    \sum_{\n\in\Nr} \frac {|\n|^{q-1}}{\bn^{2q}}<\infty. 
  \end{equation}
  In particular, \eqref{eq:slln:E=0} holds if, for some $\eps>0$,
  \begin{displaymath}
    \bn=|\n|^{1/2} \left(\log n_1\cdots\log n_r\right)^{1/(2q)+\eps}.
  \end{displaymath}
\end{corollary}
\begin{proof}
  By \eqref{sequence-if-r>1}, 
  \begin{displaymath}
    \an=\incrementbigg{|\n|^{q-1}\sum_{\k\le\n} 1}\E Z_1^{2q}
    =\prod_{i=1}^r \left( n_i^{q}-(n_i-1)^{q}\right) \E Z_1^{2q} 
  \end{displaymath}
  is of the order $|\n|^{q-1}$.
\end{proof}

Thus, assuming the existence of sufficiently high moments for the
summands (so that $q$ becomes large), 
it is possible to bring the normalization 
to $|\n|^{1/2}$ times an arbitrarily small power of $\log
n_1\cdots\log n_r$.
If $q=1$, then condition \eqref{condition-for-(2,infty)]} becomes the
condition imposed in Corollary~\ref{cor:5.1} with $\alpha=2$. 

\subsection{Stationary case and martingale dependence}
\label{sec:stat-case-mart}

Now assume that $\field Z$ are stationary in the wide sense, that is
$\E\Zn=0$ for all $\n$, $\Zn$ is square integrable, and
$\E[Z_{\n+\k}Z_{\n}]=\E[Z_\k Z_0]=R(\k)$ for all $\n,\k\in\Nr$. Then
\begin{displaymath}
  \sum_{\n\in\Nr} \frac{|R(\n)|}{|\n|^2}
  (\log n_1)^2\cdots(\log n_r)^2<\infty 
\end{displaymath}
ensures the validity of the ergodic theorem for multiple sums meaning
the almost sure convergence of $\Sn/|\n|$ to a possibly random
limit, see \cite{gap81,kles81}.

Another possible generalization for the dependent case relies on the
martingale property of the field $\field S$ meaning that the
conditional expectation of $\Sn$ given the $\sigma$-algebra generated
by $\Sk$ with $\k\leq\m$ equals the value of the field at the
coordinatewise minimum of $\n$ and $\m$, see \cite{zak81}. Then 
$\Zn=\Delta[\Sn]$, $\n\in\Nr$, is the array of multivariate martingale
differences. The following result is the martingale version of
Corollary~\ref{cor:5.1}.

\begin{theorem}
\label{T5.1}
Let $\field b$ be monotonic and \eqref{bn->infty} hold.
If $\field Z$ is such that $\field S$ is a multiparameter martingale,
  $\E|\Zn|^\alpha<\infty$ for $\alpha\in(1,2]$ and all $\n$, and
  \begin{equation}
    \label{eq:5-second}
    \sum_{\n\in\Nr}\frac {\E|\Zn|^\alpha}{\bn^\alpha}<\infty,
  \end{equation}
  then \eqref{eq:slln:E=0} holds.
\end{theorem}
\begin{proof}
  Let the sets $A_t$, $t\ge0$, and multiindices $\n_t$, $t\ge0$, be
  defined as in the proof of Theorem~\ref{alpha>=2}.  Fix $t\ge0$ and
  let random variables $Z_k^*$, $\k\le\n_t$, and $S_n^*$, $\n\le\n_t$,
  be the same as in the proof of Theorem~\ref{alpha>=2}.  
  The multi-index generalization
  of Doob's maximal inequality (see Wichura~\cite{wic69}) yields that 
  \begin{displaymath}
    \E{\max_{\n\in A_t}|\Sn|^{\alpha}} \le \E{\max_{\n\le\n_t} |\Sn^*|^{\alpha}}
    \le \left(\frac{\alpha}{\alpha-1}\right)^{\alpha r}\E{|S_{\n_t}^*|^{\alpha}}.
  \end{displaymath}
  Since $\Sn^*$, $\n\le\n_t$, is a martingale in every coordinate of
  $\n$ when others are fixed, von Bahr--Esseen's
  inequality~\cite{Bahr-Esseen} yields that
  \begin{displaymath}
    \E{|S_{\n_t}^*|^{\alpha}}
    \le 2^r \sum_{k\le \n_t} \E |\Zn^*|^\alpha
    = 2^r \sum_{\k\in A_t} \E |\Zn|^\alpha.
  \end{displaymath}
  Combining the latter two inequalities, we complete the proof by
  referring to \cite[Th.~8.2]{kles14}.
\end{proof}

\begin{remark}
  \label{about-alpha=1}
  The case $\alpha=1$ can also be treated in the martingale setting
  and the conditions involve the moments $\E[|Z_\n|\log^+|Z_\n|]$.
  This is explained by the different form of the Doob's inequality for
  first moments of multiple sums, see \cite{wic69}, that includes a
  logarithmic term.
\end{remark}

\section{Ergodic theorems for random measures and point processes}
\label{sec:an-ergodic-theorem}

Let $S(\cdot)$ be a random measure defined on Borel sets in $\Rr$, see
\cite[Def.~9.1.VI]{dal:ver08}.  The random measure is called
\emph{stationary} if $S(\cdot)$ coincides in distribution with $S(\cdot+\x)$
for each translation $\x\in\Rr$.  In this case, $\E S(\cdot)$ (if finite)
is a translation invariant Borel measure on $\Rr$ and so is proportional to
the Lebesgue measure $\lambda$. The random measure $S$ is called
\emph{completely random} if its values on disjoint sets are
independent.

\subsection{Stationary random measures}
\label{sec:stat-rand-meas}

Denote by $\I$ the semi-open unit cube $(0,1]^r$ in $\Rr$.
The ergodic theorem \cite[Th.~12.2.IV]{dal:ver08} for stationary
random measures establishes that $S(A_n)/\lambda(A_n)$ converges
almost surely and in $L^1$ to $\E[S(\I)|\sT]$, where $\sT$ is the
$\sigma$-algebra of translation invariant events and $\{A_n,n\geq1\}$
is any convex averaging sequence. The latter means that $A_n$,
$n\ge1$, are nested convex sets such that the diameter of the largest
ball inscribed in $A_n$ tends to infinity. 

More general averaging sequences $\{A_n,n\ge1\}$ were considered in
\cite{temp72}. While it is rather difficult to handle general non-nested
sequences of sets, the following result gives an ergodic theorem for
the case of $A_n=[0,\x]$ being (non-nested) rectangles in $\R_+^r$.

\begin{theorem}
  \label{thr:z}
  Let $S$ be a stationary random measure such that
  \begin{displaymath}
    |S(\I)|(\log^+|S(\I)|)^{r-1}
  \end{displaymath}
  is integrable. Then
  \begin{displaymath}
    \frac{S([0,\x])}{\lambda([0,\x])}\to \E[S(\I)|\sT] \quad
    \text{a.s. as }\; \min(x_1,\dots,x_r)\to\infty\,.
  \end{displaymath}
\end{theorem}
\begin{proof}
  The value of $S([0,\x])$ can be bounded above and below using the
  integrals of $S(\I+u)$ over $u$ from $[0,\x]$ and $[0,\x-(1,\dots,1)]$,
  respectively. The convergence of these integrals is ensured by 
  the Zygmund multivariate ergodic theorem \cite{zyg51}, see also
  \cite[Th.~10.12]{kal02}. 
\end{proof}

In the discrete version of Theorem~\ref{thr:z}, the min-convergence
can be replaced by the max-convergence, that is 
\begin{equation} 
  \label{eq:5}
  \frac{S([0,\n])}{|\n|}\to \E[S(\I)|\sT] \quad
  \text{a.s. as }\; |\n|\to\infty\,,
\end{equation}
see \cite[Prop.~A.2]{kles14}. In the following, assume that the random
measure is ergodic (or metrically transitive), so that $\sT$ is
trivial and we obtain the unconditional expectation as the limit. This
is the case if $S$ is completely random. 

We say that the random measure $S$ satisfies the strong law of large
numbers with normalization $\field b$ if
\begin{equation}
  \label{eq:slln-measure}
  \frac{S([0,\n])-\E S([0,\n])}{\bn}\to 0 \quad
  \text{a.s. as }\; |\n|\to\infty\,.
\end{equation}

Smythe's strong law of large numbers for multiple sums~\cite{smy73}
implies that~\eqref{eq:5} holds for a stationary completely random
measure under the same assumption on the logarithmic moment of $S(\I)$
as in Theorem~\ref{thr:z}. If also $\E S(A)=0$ for all Borel $A$,
Corollary~\ref{cor:5.1} yields that, if $\E |S(\I)|^\alpha<\infty$
with $\alpha\in[1,2]$, then~\eqref{eq:slln-measure} holds if $\sum
\bn^{-\alpha}$ converges. By Corollary~\ref{cor:equal-mean},
\eqref{eq:slln-measure} also holds if $S(A)$ is $(2q)$-integrable with
integer $q\geq1$ and \eqref{condition-for-(2,infty)]} is satisfied,
which is then the Brunk--Prohorov theorem for stationary completely
random measures.

The strong law of large numbers for second order stationary random
fields with the discrete parameter (see Klesov~\cite{kles81},
Gaposhkin~\cite{gap81}) can be also applied in this setting.

\subsection{Random measures generated by marked point processes}
\label{sec:rand-meas-gener}

An important family of random measures is generated by marked point
processes.  Let $\eta=\{(\x_i,y_i),i\geq1\}$ be a \emph{marked point
  process} in $\Rr$, that is $\eta$ can be viewed as a locally finite
set of pairs $(\x_i,y_i)$, where $\x_i$ is a point in $\Rr$ and $y_i$ is
a real number regarded as the mark of $\x_i$, see
\cite[Def.~6.4.I]{dal:ver03}. A marked point process can be also
defined as a non-marked point process in $\Rr\times\R$.  
Let
\begin{displaymath}
  S(A)=\sum_{i:\; \x_i\in A} y_i
\end{displaymath}
be the sum of marks for the points located in a Borel set $A$. So
defined random measure $S$ is completely random if and only if
$\{\x_i,i\geq1\}$ form a Poisson process and the conditional
distribution of $y_i$ is specified by a kernel $P(\cdot|\x_i)$, see
\cite[Prop.~10.1.VI]{dal:ver08}. Write $\E(y^2|\x)$ for the second
moment of $y$ sampled from $P(\cdot|\x)$ (assuming this moment is
finite) and denote by $\Lambda$ the
intensity measure of the Poisson process $\{\x_i,i\geq1\}$.

The strong law of large numbers of the type \eqref{eq:slln-measure}
follows from the strong law of large numbers for the partial sums of
the discrete random field $\Zn=S(\Cn)$, where 
\begin{displaymath}
  \Cn=\I+\n-(1,\dots,1),\quad \n\in\Nr\,,
\end{displaymath}
are cubes partitioning $\Rr$.

\begin{remark}
  The above construction of $S(A)$ follows the modern theory of point
  processes, see \cite{dal:ver08}.  Instead of using the definition of
  a marked point process, Smythe~\cite{smy75} considered a point
  process $\{\x_i,i\geq1\}$ and sequences of i.i.d. random variables
  $\{y_{i,\n};i\ge1\}$ which allocate marks to the points lying inside
  $\Cn$; he assumed that the marks are independent of the points
  $\{\x_i,i\geq1\}$ and between different $\Cn$. Then $\Zn$ becomes the
  sum of the marks for points $\x_i\in\Cn$. This situation is a special
  cases of our setting. 
\end{remark}

In view of centering involved in \eqref{eq:slln-measure}, it is
possible to assume that $\E(y|\x)=0$ for all $\x$. Then 
\begin{displaymath}
  \E S(\Cn)=\E \sum_{i,j\geq1}
  y_iy_j\one_{\x_i\in\Cn}\one_{\x_j\in\Cn}
  =\int_{\Cn} \E(y^2|\x) \Lambda(d\x).  
\end{displaymath}
Corollary~\ref{cor:5.1} with $\alpha=2$ 
yields that \eqref{eq:slln-measure} holds with
a monotonic $\{b_\n,\n\in\Nr\}$ satisfying \eqref{bn->infty} if 
\begin{displaymath}
  \sum_{\n\in\Nr} b_\n^{-2} \int_{\Cn} \E(y^2|\x) \Lambda(d\x) <\infty.
\end{displaymath}
In particular, if the marks are independent of the positions, then this
condition turns into convergence of the series $\sum b_\n^{-2}
\Lambda(\Cn)$. A similar reasoning applies for moments of order $(2q)$
involved in the Brunk--Prohorov strong law of large numbers, but the
conditions become less transparent. 

\section{The Kronecker lemma for multiple series}
\label{sec:kron-lemma-mult-2}

In fact, a generalization of Kronecker's lemma for $\Nr$ is
valid only for nonnegative terms and thus the proof of
\cite[Th.~2.1.1]{smy75} is not complete.  We fill the gap in its proof
below by proving a generalization of Kronecker's lemma for multiple
sums. 

\begin{lemma}
  \label{kron-gen}
  Assume that $\field x$ are non-negative numbers, and $\field b$ is
  monotonic and tends to infinity as $\n\to\infty$ (respectively, as
  $|\n|\to\infty$).  If the series
  \begin{displaymath}
    \sum_{\n\in\Nr} \frac {\xn}{\bn}
  \end{displaymath}
  converges, then
  \begin{equation}
    \label{eq:1}
    \frac1{\bn}\sum_{\k\le\n} \xk\to0 \quad \text{as }\; \n\to\infty\;
    (\text{resp.}\; |\n|\to\infty).
  \end{equation}
\end{lemma}
\begin{proof}
  The statement for the convergence as $|\n|\to\infty$
  coincides with~\cite[Prop.~A.9]{kles14}.  The case of
  convergence as $\n\to\infty$ is literally the same until the very
  last line of the proof, where one refers to $b_\n\to\infty$ as
  $\n\to\infty$ rather than to $b_\n\to\infty$ as $|\n|\to\infty$,
  see also \cite[Lemma~2.3.1]{khos02}.

  For a direct argument, use the non-negativity condition to deduce
  that, for large $n$,  
  \begin{displaymath}
    \frac1{\bn}\sum_{\k\le\n} \xk
    =\frac1{\bn}\sum_{\k\le\n} \frac{\xk}{\bk}\bk
    \leq \frac1{\bn}\sum_{|\k|\le n_0} \xk
    +\sum_{|\k|>n_0} \frac{\xk}{\bk}\bk\,. \qedhere
  \end{displaymath}
\end{proof}

\begin{remark}
  If $r=1$, then the non-negativity condition in Lemma~\ref{kron-gen}
  is not needed, since it coincides with the standard Kronecker lemma
  in this case.  
\end{remark}

\begin{remark}
  The non-negativity assumption on $\xn$ is essential as the following
  two-dimensional example shows. Let $\bn=n_1n_2$ for $r=2$. Then
  $\Delta[\bn]=1$ for all $\n$. Define
  \begin{displaymath}
    \x_{k_1k_2}=
    \begin{cases}
      -k_2, & k_1=1, \\ 2k_2, & k_1=2, \\ 0, & k_1>2.
    \end{cases}
  \end{displaymath}
  For all $n_1\ge2$ and $n_2\ge1$, we have
  \begin{displaymath}
    \sum_{k_1=1}^{n_1}\sum_{k_2=1}^{n_2}
    \frac{\x_{k_1k_2}}{\b_{k_1k_2}}=0,
  \end{displaymath}
  whence the double series
  \begin{displaymath}
    \sum_{k_1,k_2=1}^\infty \frac{\x_{k_1k_2}}{\b_{k_1k_2}}
  \end{displaymath}
  converges to zero for any reasonable definition of the convergence
  in $\N^2$.  However, the sequence
  \begin{displaymath}
    \frac{1}{\b_{n_1n_2}}
    \sum_{k_1=1}^{n_1}\sum_{k_2=1}^{n_2} \x_{k_1k_2}
    =\frac {n_2+1}{2n_1},
    \qquad n_1\ge2, \ n_2\ge1,
  \end{displaymath}
  has no limit for any reasonable definition of the convergence of
  $\n$ to infinity. 
\end{remark}

\section*{Acknowledgements}

This work has been supported by the Swiss National Science Foundation
Scopes Programme Grant IZ7320\_152292. The comments of the referees
helped to eliminate occasional misprints and have led to numerous
improvements in the presentation.


\begin{thebibliography}{10}

\bibitem{brun48}
H.~D. Brunk.
\newblock The strong law of large numbers.
\newblock {\em Duke Math. J.}, 15:181--195, 1948.

\bibitem{dal:ver03}
D.~J. Daley and D.~Vere-Jones.
\newblock {\em An Introduction to the Theory of Point Processes. Vol. {I}:
  Elementary Theory and Methods}.
\newblock Springer, New York, 2 edition, 2003.

\bibitem{dal:ver08}
D.~J. Daley and D.~Vere-Jones.
\newblock {\em An Introduction to the Theory of Point Processes. Vol. {II}:
  General Theory and Structure}.
\newblock Springer, New York, 2 edition, 2008.

\bibitem{dhar:fab:jog68}
S.~W. Dharmadhikari, V.~Fabian, and K.~Jogdeo.
\newblock Bounds on the moments of martingales.
\newblock {\em Ann. Math. Statist.}, 39:1719--1723, 1968.

\bibitem{faz:kles00}
I.~Fazekas and O.~Klesov.
\newblock A general approach to the strong laws of large numbers.
\newblock {\em Teor. Veroyatnost. i Primenen.}, 45:568--583, 2000.

\bibitem{gap81}
V.~F. Gaposhkin.
\newblock Multiparametric strong law of large numbers for homogeneous random
  fields.
\newblock {\em Uspekhi Mat. Nauk}, 36(6(222)):197--198, 1981.

\bibitem{kal02}
O.~Kallenberg.
\newblock {\em Foundations of Modern Probability}.
\newblock Springer-Verlag, New York, second edition, 2002.

\bibitem{khos02}
D.~Khoshnevisan.
\newblock {\em Multiparameter Processes}.
\newblock Springer-Verlag, New York, 2002.

\bibitem{kles14}
O.~Klesov.
\newblock {\em Limit Theorems for Multi-indexed Sums of Random Variables},
  volume~71.
\newblock Springer, Heidelberg, 2014.

\bibitem{kles81}
O.~I. Klesov.
\newblock The strong law of large numbers for homogeneous random fields.
\newblock {\em Teor. Veroyatnost. i Mat. Statist.}, 25:29--40, 166, 1981.

\bibitem{kles98}
O.~I. Klesov.
\newblock A new method for the strong law of large numbers for random fields.
\newblock {\em Theory Stoch. Process.}, 4(1-2):122--128, 1998.

\bibitem{lag09}
Z.~A. Lagodowski.
\newblock Strong laws of large numbers for {$\Bbb B$}-valued random fields.
\newblock {\em Discrete Dyn. Nat. Soc.}, pages Art. ID 485412, 12, 2009.

\bibitem{petrov95}
V.~V. Petrov.
\newblock {\em Limit Theorems of Probability Theory}, volume~4 of {\em Oxford
  Studies in Probability}.
\newblock The Clarendon Press, Oxford University Press, New York, 1995.

\bibitem{proh50}
Yu.~V. Prohorov.
\newblock On the strong law of large numbers.
\newblock {\em Izvestiya Akad. Nauk SSSR. Ser. Mat.}, 14:523--536, 1950.

\bibitem{smy73}
R.~T. Smythe.
\newblock Strong laws of large numbers for {$r$}-dimensional arrays of random
  variables.
\newblock {\em Ann. Probability}, 1:164--170, 1973.

\bibitem{smy75}
R.~T. Smythe.
\newblock Ergodic properties of marked point processes in {$R^{r}$}.
\newblock {\em Ann. Inst. H. Poincar\'e. Sect. B (N.S.)}, 11:109--125, 1975.

\bibitem{son:than13}
T.~C. Son and D.~H. Thang.
\newblock The {B}runk-{P}rokhorov strong law of large numbers for fields of
  martingale differences taking values in a {B}anach space.
\newblock {\em Statist. Probab. Lett.}, 83:1901--1910, 2013.

\bibitem{temp72}
A.~A. Tempel'man.
\newblock Ergodic theorems for general dynamical systems.
\newblock {\em Trudy Moskov. Mat. Ob\v s\v c.}, 26:94--132, 1972.
\newblock In Russian.

\bibitem{Bahr-Esseen}
B.~von Bahr and C.-G. Esseen.
\newblock Inequalities for $r$-th absolute moment of a sum of random variables,
  $1\le r\le 2$.
\newblock {\em Ann. Math. Statist.}, 36:299--303, 1965.

\bibitem{wic69}
M.~J. Wichura.
\newblock Inequalities with applications to the weak convergence of random
  processes with multi-dimensional time parameters.
\newblock {\em Ann. Math. Statist.}, 40:681--687, 1969.

\bibitem{zak81}
M.~Zakai.
\newblock Some classes of two-parameter martingales.
\newblock {\em Ann. Probab.}, 9:255--265, 1981.

\bibitem{zyg51}
A.~Zygmund.
\newblock An individual ergodic theorem for non-commutative transformations.
\newblock {\em Acta Sci. Math. Szeged}, 14:103--110, 1951.

\end{thebibliography}

\end{document}